\documentclass{article}
\usepackage{latexsym}
\usepackage{amssymb}
\usepackage{amsthm} 
\usepackage{amsmath}
\usepackage[]{color}
\usepackage{enumitem}
\usepackage{hyperref}

\newtheorem{theorem}{Theorem}

\newtheorem{corollary}{Corollary}

\newtheorem{lemma}{Lemma}
\newtheorem{question}{Question}

\newtheorem{definition}{Definition}
\newtheorem{example}{Example}

\def\1{\medskip \noindent}
 
 \def\Z{{\mathbb Z}}

 \def\N{{\mathbb N}}

\newcommand\inv{{^{-1}}}
\def\st{^{st}}
\def\th{^{th}}
\def\nd{^{nd}}
\def\rd{^{rd}}
\def\Det{{\rm Det}}
\def\Dist{{\rm Dist}}
\def\aut{{\rm Aut}}

\begin{document}

\title{\bf The Cost of 2-Distinguishing Hypercubes}
\author{\Large Debra L. Boutin \\ Hamilton College \\ Clinton, NY  13323\\ {\tt dboutin@hamilton.edu}}

\maketitle

\begin{abstract}    A graph $G$ is said to be {\it $2$-distinguishable} if there is a labeling of the vertices with two labels so that only the trivial automorphism preserves the labels.  The minimum size of a label class, over all 2-distinguishing labelings, is called the {\it cost of $2$-distinguishing}, denoted by $\rho(G)$. For $n\geq 4$ the hypercubes $Q_n$ are 2-distinguishable, but the values for $\rho(Q_n)$ have been elusive, with only bounds and partial results previously known. This paper settles the question. The main result can be summarized as: for $n\geq 4$, $\rho(Q_n) \in \{1+\lceil \log_2 n \rceil, 2 + \lceil \log_2 n\rceil\}$.  Exact values are be found using a recursive relationship involving a new parameter $\nu_m$, the smallest integer for which $\rho(Q_{\nu_m})=m$.  The main result is 

\begin{gather*} 4\leq n \leq 12\Longrightarrow \rho(Q_n)=5, \text{ and } 5\leq m \leq 11 \Longrightarrow \nu_m=4; \\
\text{ for } m\geq 6, \rho(Q_n) = m \iff 2^{m-2} - \nu_{m-1} + 1 \leq n \leq 2^{m-1}-\nu_m; \\
\text{ for } n\geq 5, \nu_m = n \iff 2^{n-1} - \rho(Q_{n-1}) + 1\leq m \leq 2^{n}-\rho(Q_n).\end{gather*}

\end{abstract}



\maketitle
\section{Introduction}

\1A labeling of the vertices of a graph $G$ with the integers $1$ through $d$ is called a {\it $d$-distinguishing labeling} if no nontrivial automorphism of $G$ preserves the labels. A graph is called {\it $d$-distinguishable} if it has a $d$-distinguishing labeling. The smallest integer for which $G$ has a $d$-distinguishing labeling is called the {\it distinguishing number}, $\Dist(G)$ \cite{AC}.   Recent work shows that in many graph families, all but a few members are $2$-distinguishable.   Examples of 2-distinguishable finite graphs include hypercubes $Q_n$ with $n\geq 4$ \cite{BC}, Cartesian powers  $G^n$ for a connected graph $G\ne K_2,K_3$ and $n\geq 2$ \cite{A, IK2,KZ}, Kneser graphs $K_{n:k}$ with  $n\geq 6, k\geq 2$ \cite{AlBo2007},  $3$-connected planar graphs (with seven small exceptions)  \cite{FNT}, and for a 2-distinguishable graph $G\ne K_1, K_2$ , $\mu^{(t)}(G)$ the generalized Mycielski construction applied to $G$ \cite{BCKLPR2020a}.  Examples of 2-distinguishable infinite graphs include the  denumerable random graph \cite{IKT}, the infinite hypercube of dimension $n$ \cite{IKT}, locally finite trees with no vertex of degree 1 \cite{WZ}, and denumerable vertex-transitive graphs of connectivity 1 \cite{STW}. 

\1A label class in a distinguishing labeling of a graph is called a {\it distinguishing class}. If $G$ is 2-distinguishable, the minimum size of a distinguishing class, over all 2-distinguishing labelings of $G$, is called the {\it cost of $2$-distinguishing} $G$, denoted $\rho(G)$ \cite{Bou2008}.  This parameter is also sometimes referred to as the {\it distinguishing cost} of $G$.  Another useful graph parameter that deals with graph symmetry is the size of a smallest set of vertices whose pointwise stabilizer is trivial. This is called the {\it determining number}, $\Det(G)$. There are multiple interesting connections between $\Det(G), \Dist(G)$, and $\rho(G)$, but for the sake of efficiency, we will not explore them fully here. However, it is straightforward to show that a distinguishing class for  $G$ has trivial  pointwise stabilizer.  This immediately tells us that $\Det(G) \leq \rho(G)$. For some families of graphs, $\rho(G)$ can even be found in terms of $\Det(G)$ \cite{Bou2013a}.  We will encounter this relationship in this paper.  Note that though $\Det(G)$ is a lower bound for $\Dist(G)$, it is not always a good lower bound.  As seen  in \cite{BoIm2017},  the cost of 2-distinguishing can be an arbitrarily large multiple of the determining number.

\1The originating work for the cost of 2-distinguishing \cite{Bou2008} was in answer to a question posed by Wilfried Imrich in 2007, \lq\lq  What is the minimum number of vertices in a label class of a $2$-distinguishing labeling for the hypercube?" In the 2005 paper proving that for $n\geq 4$, $Q_n$ is 2-distinguishable,  Bogstad and Cowen \cite{BC} use smallest distinguishing classes of size $n+2$. The best result known prior to \cite{Bou2008}  was $\rho(Q_n) \approx \sqrt n$ \cite{IW}. In \cite{Bou2008} this author showed that for $n\geq 5$, $\lceil \log_2 n\rceil +1 \leq \rho(Q_n)\leq 2\lceil \log_2 n\rceil - 1$.  Though the upper bound is less than twice the lower bound, and though even the upper bound was much smaller than previously known results, having only bounds on the distinguishing cost was not satisfying.  A small amount of progress was made in 2013 when a few exact values were found: for $m\geq 5$ and $n\in \{2^{m-1}-2, 2^{m-1}-1, 2^{m-1}\}$, $\rho(Q_n)=m+1$ \cite{Bou2013b}. 

\1It is fruitful to consider $Q_n$ as the $n\th$ Cartesian power of $K_2$, denoted $K_2^n$.  In 2013, this author shows that for prime graphs $H$ meeting mild hypotheses, we get $\rho(H^k) \in \{\Det(H^k), \Det(H^k)+1\}$ \cite{Bou2013a}, which is useful because results on $\Det(H^k)$ are given in \cite{Bou2009}. However, one of the mild hypotheses that is critical for finding $\rho(H^k)$ is that $H$ have at least 3 vertices. So the technique used for more general Cartesian powers is not applicable to $Q_n$. In this paper, the technique  from \cite{Bou2013a} is greatly refined into a tool for proving exact values for the cost of distinguishing hypercubes.  More specifically, if $n$ is greater than but not \lq\lq too close" to $2^{m-2}$, and less than but not \lq\lq too close" to $2^{m-1}$, then $\rho(Q_n)=m$. For such $n$, we can then solve to get $\rho(Q_n) = 1+\lceil \log_2 n\rceil$.  Further, if $n$ is \lq\lq too close" to $2^{m-1}$, we get $\rho(Q_n) = m+1$, or equivalently $\rho(Q_n) = 2+\lceil \log_2 n\rceil$.  Since by \cite{Bou2009}, $\Det(Q_n) = 1 + \lceil \log_2 n \rceil$, this paper concludes that for $n\geq 4$, $\rho(Q_n) \in \{\Det(Q_n), \Det(Q_n)+1\}$. 

\1To achieve exact values we want for the distinguishing cost for hypercubes, the big question becomes, \lq\lq How close is too close?"  We define \lq\lq too close" to $2^{m-2}$ as $\nu_{m-1}$ and \lq\lq too close" to $2^{m-1}$ as $\nu_m-1$.  That is, for $m\geq 5$, we produce a sequence of integers $\{\nu_m\}$  with the property that if $2^{m-2}-\nu_{m-1}+1 \leq n \leq 2^{m-1} - \nu_m$, then $\rho(Q_n)=m$.  This provides a natural recursive relationship between $\{\rho(Q_n)\}$ and $\{\nu_m\}$, which will allow us to find the value of each for any $n\geq 4$ or $m\geq 5$.

\1The paper is organized as follows.  In Section~\ref{sec:charmat}, for any subset of vertices in $Q_n$ we define the {\it characteristic matrix} of the set, define what it means for a matrix to be {\it asymmetric}, and show that set is a distinguishing class if and only if its characteristic matrix is asymmetric.  This section also gives and proves  essential properties of symmetries of characteristic matrices.  We prove the existence of asymmetric matrices of size $m{\times} m, m{\times} (m-1)$, and $m{\times}\lfloor \frac{m}{2} \rfloor$ in Section~\ref{sec:basicasymm}. In Section~\ref{sec:comp}, we prove a 2016 theorem from Richard Stong \cite{RS} giving criteria for when the asymmetric nature of one matrix guarantees the asymmetry of another.  We finish creating the families of asymmetric matrices that we need to give an upper bound on the distinguishing cost in Section~\ref{sec:const}, while a lower bound is proved in Section~\ref{sec:lowerbd}.  In Section~\ref{sec:distcost}, we pull the previously established bounds together to state our final results and to examine its consequences. In Section~\ref{sec:DetNum} we connect the results to determining numbers.  Finally, in Section~\ref{sec:quest} we provide some open problems. Appendix~\ref{ap:small} provides existence proofs for certain small dimension asymmetric matrices that were claimed in Section~\ref{sec:const}.


\section{Characteristic matrices}\label{sec:charmat}

\1In this paper all matrices will be binary.  This fact will be mentioned explicitly at times, but not always.

\begin{definition} \rm Let $S=\{V_1,\ldots, V_m\}$ be an ordered set of vertices of $Q_n$, each written as a binary string of length $n$.   Define $X$ to be the $m{\times} n$ matrix whose $i^{th}$ row contains the coordinates for $V_i$. Call $X$ a {\it characteristic matrix} of $S$.\end{definition}   

\1We want to explicitly show the correspondence between automorphisms of $Q_n$ and actions on the characteristic matrix.  As is usual, we will consider each coordinate of a binary string representing $V\in V(Q_n)$ as an element of $\Z_2$. 

\begin{theorem} \rm \cite{HIK} The map $\varphi\in{\rm Aut}(Q_n)$ if and only if there is a permutation $\pi\in S_n$ and for $1\leq i \leq n$, isomorphisms $\psi_i:\Z_2 \to \Z_2$ so that $$\varphi(v_1\cdots v_n)=(\psi_{\pi^{-1}(1)}(v_{\pi^{-1}(1)}) \cdots, \psi_{\pi^{-1}(n)}(v_{\pi^{-1}(n)})).$$  Denote $\varphi$ as $(\pi, \{\psi_i\})$.\end{theorem}

\1We wish to translate the effect of $\varphi = (\pi,\{\psi_i\})\in {\rm Aut}(Q_n)$ on our ordered subset $S\subseteq V(Q_n)$ to its effect on the characteristic matrix $X$.  Since $\pi$ permutes the coordinates of each vertex, we can consider $\pi$ as a permutation of the columns of $X$. Since each $\psi_i$ applies one of the two automorphisms of $\Z_2$ to coordinate $i$ of each vertex, we can consider $\psi_i$ as applying to each entry in column $i$.  Thus given $\varphi = (\pi,\{\psi_i\})\in {\rm Aut}(Q_n)$, we can consider its action on $X$.  Further given any $\pi\in S_n$ and any $\{\psi_i\}\in (\aut(\Z_2))^n$, there is a natural action of $\varphi=(\pi, \{\psi_i\})$ on $X$ that corresponds to a unique $\varphi \in \aut(Q_n)$.  When referring to the characteristic matrix we will call $\varphi$ a {\it permaut} of the columns of $X$.  Denote the result of applying permaut $\varphi$ to the columns of $X$ by $X^\varphi$. By the definition of the action of $\varphi$ on $X$, $X^\varphi$ is the characteristic matrix of the ordered subset $\varphi(S)$.  Thus  $\varphi\in {\rm Aut}(Q_n)$ preserves the set $S$ if and only if $X$ and $X^\varphi$ have the same set of rows, possibly permuted. Denote the result of applying a  permutation $\sigma \in S_n$ to the rows of $X$  by $X_\sigma$. 

\begin{definition} \rm Let $X$ be a a binary matrix. If there exists a column permaut $\varphi$ and a row permutation $\sigma$ so that $X_\sigma = X^\varphi$, then we say that $(\sigma,\varphi)$ is a {\it symmetry} of $X$. If the only symmetry of $X$ is the trivial symmetry $(id,id)$, then we say that $X$ is {\it asymmetric}.\end{definition} 

\1A subset $S\subseteq V(Q_n)$ is a distinguishing class if and only if the only automorphism that preserves $S$ setwise is the trivial automorphism.  Thus, we can write the criterion for a subset $S$ to be a distinguishing class for $Q_n$ in terms of its characteristic matrix in the following way.

\begin{theorem}\label{thm:distasymm} \rm The ordered subset $S\subseteq V(Q_n)$ is a distinguishing class for $Q_n$ if and only if its characteristic matrix is asymmetric.\end{theorem}

\1Below are useful tools for working with characteristic matrices, their column permauts, and their row permutations.  Since the proofs are straightforward, for efficiency, we will mostly point the way to the proofs.

\begin{lemma}\label{lem:Obs} \rm Let $X$ be a binary matrix with column permauts $\varphi,\omega$  and row permutations  $\sigma, \tau$. 

\begin{enumerate}[label=\alph*)] 

\item Using map composition we easily get that \begin{enumerate}[label=(\roman*)]

\item $(X^\varphi)^\omega = X^{\omega \varphi}$ \item  $(X_\sigma)_\tau = X_{\tau \sigma}$ \item $(X_\sigma)^\alpha = (X^\alpha)_\sigma$. 

\end{enumerate}

\item \label{lem:ObsPermAuts:1}  $\sigma$ preserves the property of two columns being (or not being) isomorphic. 

\1That is,  columns $i$ and $j$ of $X$ are isomorphic if and only if  columns $i$ and $j$ of $X_\sigma$ are isomorphic.

\item \label{lem:ObsPermAuts:2} $\varphi$ preserves the distinctness (or nondistinctness) of rows.  

\1That is, row $i$ and $j$ of $X$ are distinct if and only if  rows $i$ and $j$ of $X^\varphi$ are distinct.

\item \label{lem:ObsPermAuts:3} $X$ is asymmetric if and only if $X^\omega$ is asymmetric if and only if $X_\tau$ is asymmetric.  

\1This is true because $(\sigma,\varphi)$ is a symmetry of $X$  if and only if $(\sigma,\varphi)$ is a symmetry of $X_\tau$ if and only if $(\sigma, \omega\varphi \omega\inv)$ is a symmetry of $X^\alpha$.

\end{enumerate}\end{lemma}

\1Two different orderings of a set produce characteristic matrices that differ only by a row permutation.  By Lemma~\ref{lem:Obs}\ref{lem:ObsPermAuts:3}, either both are asymmetric or neither is.  Thus we needn't worry about the order of the set $S$.  

\begin{lemma} \label{lem:assump} \rm Let $X$ be an asymmetric $m{\times} n$ binary matrix.  Then $m<2^n$, $n<2^{m-1}$, $X$ has no pair of equal rows, and $X$ has no pair of isomorphic columns.\end{lemma}

\begin{proof} \1If a $m\times n$ binary matrix $X$ has two equal rows, those rows can be transposed by $\sigma$ without changing the matrix.  Thus $X$ has the nontrivial symmetry $(\sigma, id)$.  Similarly if $X$ has two isomorphic columns there is a nontrivial symmetry $(id, \varphi)$ of $X$. If $m>2^n$ then $X$ must have two equal rows, and if $n>2^{m-1}$  then $X$ must have two isomorphic columns.  Further, if $X$ has distinct rows and $m=2^n$, since by  Lemma~\ref{lem:Obs}\ref{lem:ObsPermAuts:2} any permaut $\varphi$ preserves the distinctness of the rows of $X$,  $X^\varphi$ also contains $2^n$ distinct rows.  But there are only $2^n$ distinct binary strings of length $n$.  So $X$ and $X^\varphi$ have the same set of rows.  Thus there is a row permutation $\sigma$ so that $X^\varphi=X_\sigma$, and $X$ has a nontrivial symmetry. A similar argument shows that if $X$ has $2^{m-1}$ non-isomorphic columns, then $X$ has a nontrivial symmetry. \end{proof}

\begin{definition} \rm Let $X$ be a binary matrix.  The {\it weight} of a row or column of $X$ is the number of ones it contains. A column of length $m$ with at most $\lfloor \frac{m}{2} \rfloor$ ones is called a {\it low weight column}. A matrix $X$ is said to be {\it low weight} if each of its columns is low weight.  A column (respectively matrix) can be called {\it strictly low weight} if its weight is strictly less than $\frac{m}{2}$ (respectively all matrix columns have weight  strictly less than $\frac{m}{2}$).  We will use the term {\it high weight} to denote the property of not being low weight.\end{definition}

\1Given a  high weight matrix $Y$, there is a permaut $\alpha$ so that $Y^\alpha$ is a low weight matrix.  Again by Lemma~\ref{lem:Obs}\ref{lem:ObsPermAuts:3}, either both of $Y$ and $Y^\alpha$ are asymmetric, or neither is. Thus for the remainder of this paper we will (mostly) restrict our attention to low weight binary matrices. 

 \begin{lemma}\label{lem:Rules} \rm  Let $X$ be a binary $m
{\times}n$ matrix. Suppose that $(\sigma, \varphi)$ is a symmetry of $X$. Let $\varphi = (\pi,\{\psi_i\})$.  
\begin{enumerate}[label=\alph*)]

\item \label{lem:Rules:1}  The permaut $\varphi$ can only permute columns with the same weight.  That is, if column $i$ has weight $k$ then so does column $\pi(i)$ after applying $\psi_i$. 

Since a row permutation only rearranges, but does not change, the elements of each column, $\sigma$ preserves the number of ones in each column. Since $X_\sigma = X^\varphi$, $\varphi$ must preserve the number of  ones in each column as well. 

\item \label{lem:Rules:2} If $X$ is low weight, then $\varphi$ can only provide nontrivial automorphisms on columns of $X$ that do not have weight precisely $\frac{m}{2}$. That is, if $X$ is low weight and $\psi_i$ is nontrivial, then column $i$ has weight precisely $\frac{m}{2}$.  Futher, if  $X$ is strictly low weight, then $\varphi = \pi$, a permutation of the columns.

Applying the nontrivial automorphism of $\Z_2$ to a strictly low weight column produces a high weight column.  Since by~\ref{lem:Rules:1} above, $\varphi$ can only permute columns of the same weight and all columns of $X$ are low weight, if the weight of column $i$ is not $\frac{m}{2}$, then  $\psi_i$ must be trivial.

\item \label{lem:Rules:3} If $X$ is strictly low weight, then $\sigma$ can only permute rows with the same weight.  That is, if $\sigma(i){=}j$, then row $i$ and row $j$ have the same weight.

By~\ref{lem:Rules:2} above, since $X$ is strictly low weight,  $\varphi = \pi$ is a simple permutation of the columns of $X$.  Thus $\varphi$ rearranges, but does not change, the elements of each row, $\varphi$ preserves the number of ones in each row. Since $X_\sigma = X^\varphi$,  $\sigma$ must also preserve the weight of each row.

\item \label{lem:Rules:4} If $X$ is strictly low weight, and if each row of $X$ has strictly fewer than $\frac{n}{2}$ ones, then $X$ is asymmetric if and only if its transpose is asymmetric. 

By assumption on $X$, both $X$ and $X^T$ are strictly low weight binary matrices.  Thus, by \ref{lem:Rules:2} above, for each of $X$ and $X^T$ the only symmetries are strict permutations of the rows and strict permutations of the columns.  Since $X$ and $X^T$ exchange the roles of rows and columns, given a symmetry $(\sigma, \varphi)$ of $X$ we an exchange the roles of the row permutation and the column permutation to get a symmetry $(\varphi, \sigma)$ of $X^T$.\end{enumerate}\end{lemma}

\begin{lemma}\label{lem:iff} \rm Let $X$ be a binary matrix and $(\sigma,\varphi)$ a symmetry of $X$. If $X$ has no identical rows and no isomorphic columns, then $\sigma$ is the identity row permutation if and only if $\varphi$ is the identity column permaut.\end{lemma}

\begin{proof} Let $\varphi=(\pi,\{\psi_i\})$.  Suppose that $\sigma$ is trivial.  Then $X = X_\sigma = X^\varphi$.  Suppose that $\pi(i)=j$ with $i\ne j$.  Then by definition of $\varphi$, column $i$ of $X$ is isomorphic to column $j$ of $X^\varphi=X$.  This contradicts the choice of $X$.  Thus $\pi$ is the identity.  Further, since $\sigma$ is trivial, $X$ and $X_\sigma = X^\varphi$ have the same values in each position of each column.  However, since $\pi$ is trivial,  a nontrivial $\psi_i$ must change the value of each entry in column $i$.  Thus each $\psi_i$ is trivial.  Thus $\varphi$ itself is trivial.

\1Suppose that $\varphi$ is trivial.  Then $X_\sigma = X^\varphi = X$.  Thus after performing the row permutation $\sigma$ we have the same columns we started with. Suppose $\sigma(i)=j$.  Since $X$ has no identical rows,  row $i$ and row $j$ differ in some position, say $k$. After performing $\sigma$, column $k$ has a different value in positions $i$ and $j$.  But since $X^\varphi=X$, this can't happen.  Thus $\sigma$ is trivial. \end{proof}

\1The following lemmas provide tools we will use later in the paper.

\begin{lemma}\label{lem:extracolumns} \rm If $X$ is an asymmetric $m\times r$ matrix, and $Y$ is $m{\times} s$ matrix with no pair of isomorphic columns and all column weights different than the column weights in $X$, then the $m{\times} (r+s)$ concatenation $XY$ is also asymmetric.\end{lemma}

\begin{proof} Suppose $(\sigma, \varphi)$ is a symmetry of $XY$. By Lemma~\ref{lem:Rules}\ref{lem:Rules:1}, $\varphi$ preserves weight classes of columns.  Since the weights of columns of $X$ are distinct from the weights of the columns of $Y$, $\varphi$ preserves the set of columns of $X$ and of $Y$.  Thus $\varphi$ can be decomposed into $\varphi_x$, the action of columns of $X$, and  $\varphi_y$, the action on the columns of $Y$.  That is, $(XY)^\varphi = X^{\varphi_x} Y^{\varphi_y}$. Recall that $\sigma$ is defined on rows of a matrix, without regard the the number of columns.  Thus $X^{\varphi_x} Y^{\varphi_y} = (XY)^\varphi = (XY)_\sigma = X_\sigma Y_\sigma$, and therefore, $X_\sigma = X^{\varphi_x}$.  Since $X$ is asymmetric by assumption, this means that each of $\sigma$ and $\varphi_x$ is trivial.  Since $XY$ has no identical rows and no isomorphic columns by Lemma~\ref{lem:iff}, the triviality of $\sigma$ guarantees the triviality of $\varphi$.  Thus $XY$ is asymmetric.\end{proof}

\1A similar argument proves the following.

\begin{lemma}\label{lem:extrarows} \rm If $X$ is an asymmetric $k{\times} n$ matrix, and $Z$ is a $\ell \times n$ matrix with no pair of identical rows and all row weights different than the row weights in $X$, and the $(k+\ell){\times}n$ concatenation $XZ$ has no column of weight $\frac{k+\ell}{2}$, then $XZ$ is also asymmetric.\end{lemma}


\section{Some Small(ish) Asymmetric Matrices}\label{sec:basicasymm}

\1Section~\ref{sec:const} will cover the construction of many, mostly large, asymmetric binary matrices.  Here we will construct some basic, mostly small, examples on which we can  build later. We place these constructions in this early section for two reasons.  It gives us a chance the use the rules and observations from Section~\ref{sec:charmat}, and we can then use these matrices as building blocks at the beginning of Section~\ref{sec:const} without distraction.

\1Because we will wish to refer to integers within particular intervals, all numerical intervals in this paper are integer intervals.  That is for $r < s\in \Z$, $[r,s]$ is the set of integers inside the real number interval.

\begin{lemma}\label{lem:mminusone} \rm For $m\geq 5$ there exist asymmetric binary $m{\times}m$ and $m\times(m-1)$ matrices.\end{lemma}

\begin{proof} Build an asymmetric $m{\times} m$ matrix $X$ in the following way. Let the first column of $X$ have a one in its first position and zeros elsewhere.  For $j\in [2, m]$ let column $j$ have ones in positions $j-1$ and $j$, and zeros elsewhere. See Figure~\ref{fig:seed} for an example of $X$ when $m=7$.  Note that columns 2 through $m$ have weight $2$, while column 1 has weight 1.  Further, for $j<m$, row $j$ of $X$ has ones in positions $j, j+1$, while row $m$ has a one only in the final position.  

\begin{figure}[htb]
$$\begin{array}{cccccccc} 
1 & 1 & 0 & 0 & 0 & 0 & 0 \\ 
0 & 1 & 1 & 0 & 0 & 0 & 0  \\ 
0 & 0 & 1 & 1 & 0 & 0 & 0 \\
0 & 0 & 0 & 1 & 1 & 0 & 0 \\
0 & 0 & 0 & 0 & 1 & 1 & 0 \\
0 & 0 & 0 & 0 & 0 & 1 & 1 \\
0 & 0 & 0 & 0 & 0 & 0 & 1 \end{array}$$
\caption{$X$ with $m=7$}
\label{fig:seed}
\end{figure}

\1Suppose there is a symmetry $(\sigma,\varphi)$ of $X$.  Since $m\geq 5$ and each column has at most 2 ones, $X$ is a strictly low weight matrix.  Thus by Lemma~\ref{lem:Rules}\ref{lem:Rules:2}, $\varphi$ acts strictly as a permutation on the columns of $X$, and by Lemma~\ref{lem:Rules}\ref{lem:Rules:1}, it can only permute columns of the same weight.  Thus as the only column of weight 1, the first column is fixed by $\varphi$. Since the first column is fixed by $\varphi$, the first row of $X^\varphi$ has a one in its first position.  Since the only row of $X$ with a one in its first position is the first row, $\sigma$ fixes the first row.  

\1Now assume $k<m$, that $\varphi$ fixes the first $k$ columns of $X$, and that $\sigma$ fixes the first $k$ rows.  By construction, column $k+1$ of $X$ has ones in positions $k, k+1$ and row $k+1$ of $X$ has ones in positions $k+1, k+2$.  Since by assumption $\sigma$ fixes row $k$ of $X$, column $k+1$ of $X_\sigma$ has a one in position $k+1$.  The only not-yet-known-to-be-fixed column with a one in position $k+1$ is column $k+1$ of $X$.  Thus column $k+1$ of $X_\sigma$ is column $k+1$ of $X$.  Thus $\varphi$ fixes column $k+1$. Now we can conclude that row $k+1$ of $X_\sigma$ has a one in position $k+1$.  The only not-yet-known-to-be-fixed row to have a one in position $k+1$ is row $k+1$ of $X$, so $\sigma$ fixes row $k+1$. Thus by induction, $\sigma$ fixes all rows and $\varphi$ fixes all columns of $X$, and so $(\sigma,\varphi)$ is trivial. Therefore $X$ is asymmetric. 

\1Note that if we delete the $m\th$ column we still have an asymmetric matrix.\end{proof}

\1The statement of the lemma below is also true for $m\in[8, 11]$.  However, the proof would be somewhat different, and we only need the lemma as stated. 

\begin{lemma}\label{lem:halfm} \rm For $m\geq 12$, there exists an asymmetric $m{\times} \lfloor \frac{m}{2} \rfloor$ binary matrix. \end{lemma} 

\begin{proof}  Let $r=\lfloor \frac{m}{2} \rfloor$. Begin with the $r\times r$ asymmetric matrix whose construction we learned in the proof of Lemma~\ref{lem:mminusone}. We will see below how to construct an $(m-r){\times}n$ matrix $Z$ with no identical rows, all row weights different from those in $X$,  so that the $m\times n$ concatenation $XZ$ is strictly low weight.  By Lemma~\ref{lem:extrarows}, the resulting $XZ$ will be asymmetric. 

\1Suppose that $m$ is even.  Construct $Z$ so that the $i\th$ row contains zeros in positions $i, i+1, i+2$ (modulo $r$) and ones elsewhere.  As constructed, $Z$ has no equal rows, each row of $Z$ has weight $r-3$, and each column has weight $r-3$.  See Figure~\ref{fig:Z} for an example of $Z$ for $m=14$.

\begin{figure}[htb]
$$\begin{array}{cccccccc} 
0 & 0 & 0 & 1 & 1 & 1 & 1 \\ 
1 & 0 & 0 & 0 & 1 & 1 & 1  \\ 
1 & 1 & 0 & 0 & 0 & 1 & 1 \\
1 & 1 & 1 & 0 & 0 & 0 & 1 \\
1 & 1 & 1 & 1 & 0 & 0 & 0 \\
0 & 1 & 1 & 1 & 1 & 0 & 0 \\
0 & 0 & 1 & 1 & 1 & 1 & 0 \\
\phantom{a}\end{array}$$
\caption{$Z$ for $m=14$}
\label{fig:Z}
\end{figure}

\1Concatenate $X$ and $Z$ into an $m{\times} r$ matrix $XZ$ whose first $r$ rows are the rows of $X$ and whose remaining $r$ rows are the rows of $Z$.  Note that each row of $Z$ has weight $r-3$, while the row weights of $X$ are 1 and 2. Since $m\geq 12$, $r\geq 6$ and $r-3\not\in \{1,2\}$. Thus the rows weights of $Z$ are distinct from the row weights of $X$.   The columns of $XZ$ have weights $r-2$ and $r-1$, strictly less than $\frac{m}{2}=r$, and so $XZ$ is strictly low weight.  Thus by Lemma~\ref{lem:extrarows}, since $X$ is asymmetric, so is the concatenation $XZ$.

\1Suppose $m$ is odd.  Create the matrix $Z$ as above and add a row of zeros to create $Z'$.  Then $Z'$ has rows distinct rows of weight $r-3$ and $0$, which are distinct from the row weights of $X$. Again, the column weights of the concatenation, $r-1$ and $r-2$, are strictly less than $r=\lfloor \frac{m}{2}\rfloor$.  Again,  Lemma~\ref{lem:extrarows} provides the conclusion that $XZ$ is asymmetric.\end{proof}

\begin{lemma}\label{lem:halfn} \rm For $n\geq 12$, there exists an asymmetric $\lfloor \frac{n}{2}\rfloor{\times}n$ matrix. \end{lemma} 

\begin{proof} Let $s=\lfloor \frac{n}{2} \rfloor$. The construction begins with the asymmetric $s{\times} s$ matrix $X$ from the proof of Lemma~\ref{lem:mminusone}.  Let $Z$ be the  $(n-s){\times} s$ matrix described in Lemma~\ref{lem:halfm} above.  Let $Y=Z^T$ and concatenate $X$ and $Y$.   Note that the columns of the $s\times (n-s)$ matrix $Y$ are nonisomorphic and all columns have weight different than the column weights of $X$. Then Lemma~\ref{lem:extracolumns} provides the conclusion that the $\lfloor \frac{n}{2} \rfloor \times n$ matrix $XY$ is asymmetric.\end{proof}


\section{The Complement Theorem}\label{sec:comp}

\1Before we go on to the rest of our constructions we need a bit more theory.  This  section summarizes the work of Richard Stong \cite{RS} on characteristic matrices of sets of vertices of $Q_n$.  His results make the conclusions of this paper possible.

\1Let $X$ be a binary $m{\times} n$ matrix with $n< 2^{m-1}$, $m< 2^n$, no pair of columns isomorphic,  and no two rows equal. Define $Y$ to be a  $m{\times} (2^{m-1}-n)$ binary matrix whose columns are representatives of isomorphism classes not represented as columns of $X$.  Define $Z$ to be a  $(2^n-m){\times} n$ binary matrix whose rows are  binary strings of length $n$ that are not rows of $X$.  We will see in the theorems below that either all of $X,Y,Z$ have symmetry or none has symmetry.

\begin{lemma}\label{lem:XYsymm} \rm Let $X$ be a $r {\times} s$ binary matrix with $s<2^{r-1}$, no identical rows, and no isomorphic columns.  Define $Y$ to be a $r{\times} (2^{r-1}-s)$ binary matrix whose columns are representatives of the isomorphism classes of columns that are not represented as columns of $X$.  Then $X$ has symmetry if and only if $Y$ has symmetry.\end{lemma}

\begin{proof} By Lemma~\ref{lem:Obs}\ref{lem:ObsPermAuts:3}, we may assume that $X$ and $Y$ are low weight matrices for which the union of the $2^{r-1}$ columns of $X$ and $Y$ comprise all $2^{r-1}$ low weight columns of length $r$.  Concatenate $X$ and $Y$ to an $r{\times} 2^{r-1}$ binary matrix, $XY$, whose first $s$ columns are columns of $X$ and whose last $2^{r-1}-s$ columns are columns of $Y$. 

\1Suppose that $X$ has a symmetry $(\sigma, \varphi)$ with at least one of $\varphi, \sigma$ nontrivial.  Since $X$ has no identical rows and no isomorphic columns, by Lemma~\ref{lem:iff}, we can assume both $\sigma$ and $\varphi$ are nontrivial.  Since we can apply $\sigma$ to a matrix of $r$ rows regardless of the number of columns, we may apply $\sigma$ to the rows of $XY$ and get $(XY)_\sigma = X_\sigma Y_\sigma$. 

\1By Lemma~\ref{lem:Obs}\ref{lem:ObsPermAuts:2}, a row permutation preserves the property of two columns being or not being isomorphic.  Since each of the $2^{r-1}$ columns of $XY$ represents a distinct isomorphism class, and $\sigma$ does not change this distinctness, each of the $2^{r-1}$ columns of $X_\sigma Y_\sigma$ represents a distinct isomorphism class.  Further, since by assumption $X_\sigma = X^\varphi$, and by definition of $\varphi$, the columns of $X^\varphi$ represent the same isomorphism classes as the columns of $X$.  Thus,  the columns of $Y_\sigma$ represent the same isomorphism classes as the columns of $Y$.  Thus there is a permaut $\varphi'$ of $Y$ so that $Y^{\varphi'} = Y_\sigma$.   Thus $Y$ has symmetry.

\1Switching the roles of $X$ and $Y$ produces the same result. Thus $X$ has symmetry if and only if $Y$ has symmetry.\end{proof}

\1An entirely similar proof gives us the following.

\begin{lemma}\label{lem:ZXsymm} \rm  \rm Let $X$ be a $r {\times} s$ binary matrix with no identical rows and no isomorphic columns and $r<2^s$.  Define $Z$ to be a $(2^s-r){\times} s$  binary matrix whose rows are distinct and are not rows of $X$. Then $X$ has symmetry if and only if $Z$ has symmetry.\end{lemma}

\1Together Lemmas~\ref{lem:XYsymm} and~\ref{lem:ZXsymm} prove the Complement Theorem below. 

\1{\bf The Complement Theorem:} \cite{RS}  With $X,Y,Z$ as defined above, either each of $X,Y,Z$ has symmetry or none has symmetry.


\section{More Asymmetric Matrices}\label{sec:const}

By Lemma~\ref{lem:mminusone}, for every $m\geq 5$ (respectively $n\geq 4$) there is an asymmetric $m\times(m-1)$ (respectively $(n+1)\times n$) binary matrix. Now that we are assured of existence, we will find it useful to be able to refer to the smallest such values.  In particular, for  $n\geq 4$, denote by $\mu_n$  the fewest number of rows for which there is an asymmetric $\mu_n{\times} n$ matrix. Similarly, denote by $\nu_m$ the fewest number of columns for which there is an asymmetric $m{\times} \nu_m$ matrix.  It is obvious by definition that $\rho(Q_n)=\mu_n$.  However, there is a tremendous amount of symmetry in the statements and proofs and usage of $\mu_n$ and $\nu_m$, so the symmetry of the notation is natural.  To highlight this, we will use the notation $\mu_n$ until we reach our final conclusions. 

\1Lemmas~\ref{lem:halfm} and \ref{lem:halfn} give us the following corollary, which we will find useful later on.

\begin{corollary} \label{cor:halfs} \rm For $m\geq 12$, $\nu_m\leq \lfloor \frac{m}{2}\rfloor$; for $n\geq 12$, $\mu_n\leq \lfloor \frac{n}{2}\rfloor $.\end{corollary}

\1The small cases for $m$ and $n$ resist the more general proofs that we use for larger $m$ and $n$.  The proofs of results for relatively small $m$ and $n$ are more detailed than that for Theorem~\ref{thm:fromnu}, but without being additionally enlightening.  So the proofs of Lemma~\ref{lem:smalln} and \ref{lem:smallm} and Theorems~\ref{thm:smallm} and \ref{thm:smalln} have been moved to Appendix~\ref{ap:small}. 

\begin{lemma}\label{lem:smalln}  \rm For $n\in[4,12]$, $\mu_n = 5$.\end{lemma}

\begin{lemma}\label{lem:smallm} \rm For $m\in[5,11]$, $\nu_m = 4$.\end{lemma}

\begin{theorem} \label{thm:smallm} \rm  For $m \in [5,11]$ and $n\in [4,2^{m-1}-4]$, there exists an asymmetric $m{\times}n$ matrix. \end{theorem}

\begin{theorem} \label{thm:smalln} \rm  For $n \in [4,12]$ and $m\in [5,2^n-5]$, there exists an asymmetric $m{\times}n$ matrix. \end{theorem}

\1Together the lemmas and theorems listed above yield the following statements which preview the statements for more general $m$ and $n$ given in Theorem~\ref{thm:fromnu}.\vskip-.15in

$$m \in [5,11]\ \& \ n\in [\nu_m,2^{m-1}-\nu_m] \Longrightarrow  \exists \text{ an asymmetric } m{\times}n \text{ matrix.}$$ \vskip-.25in
  
$$n \in [4,12]\ \& \ m\in [\mu_n,2^n-\mu_n] \Longrightarrow  \exists \text{ an asymmetric } m{\times}n \text{ matrix.}$$ 

\begin{theorem} \label{thm:fromnu} \rm For $m\geq 12$ and $n\in [\nu_m, 2^{m-1}-\nu_m]$, there exists an asymmetric $m{\times} n$ matrix.  \end{theorem} 

\begin{proof} Let $r=\lfloor \frac{m}{2} \rfloor$.  We will prove this theorem in four steps.  First we will show it is true in the integer interval $\left[\nu_m,r-1 \right]$, then in $\left[r, m-2\right]$, then in $[m-1,2^{m-2}]$.  Finally, we will use the Complement Theorem to obtain the result for $[2^{m-2},2^{m-1}-\nu_m]$. 

\1$\bullet${\bf \boldmath $\left[\nu_m,r-1 \right]$:} Recall that for $m\geq 12$, $\nu_m \leq r$ by Corollary~\ref{cor:halfs}. However, we may assume that $\nu_m<r$.  Otherwise this interval is empty and unnecessary for this proof.  Start with an asymmetric $m{\times} \nu_m$ matrix $X$, whose existence is guaranteed by the definition of $\nu_m$.  Assume $X$ is low weight. We wish to add $j\in [1,r-\nu_m-1]$ columns to $X$ so that the resulting $m{\times} (\nu_m+j)$ matrix is asymmetric.   By Lemma~\ref{lem:extracolumns}, it is sufficient to find a low weight $m\times j$ binary matrix, $Y$,  with no isomorphic columns whose column weights are distinct from the column weights of $X$ and concatenate the result with $X$.

\1By its dimensions, the number of distinct column weights in $X$ is at most $\nu_m$. Further, the number of possible nontrivial column weights in a low weight matrix is $r$.  Thus there are at least $r-\nu_m$  low weights that do not occur for columns in $X$.   Thus for each $j\in[1, r - \nu_m-1]$ we can find $j$ nonisomorphic columns with weights different from the weights of the columns of $X$ to create a low weight matrix $Y$.  By Lemma~\ref{lem:extracolumns}, $XY$ is asymmetric. 

\1$\bullet${\bf \boldmath $\left[r,m-2 \right]$:} Using the proof of Lemma~\ref{lem:halfm} we can construct an asymmetric $m\times r$ matrix $X$ with column weights $r-1$ and $r-2$.  We wish to add $j\in [1,m-r-2]$ columns to $X$ so that the resulting $m{\times} (r+j)$ matrix is asymmetric.   By Lemma~\ref{lem:extracolumns} it is sufficient to find a low weight $m\times j$ binary matrix with no isomorphic columns whose column weights are distinct from the column weights of $X$ and concatenate the result with $X$.  Note that if $m$ is even $m-r-2 = r-2$ while if $m$ is odd $m-r-2 = r-1$. 

\1Note that since the columns of $X$ use only two low weights, there are $r-2$  distinct low weights available for a $m\times j$ matrix $Y$. Since $m\geq 12, r\geq 6$ and so $r-2>1$.  In particular there is some unused nontrivial low weight.  Thus with $r-2$ unused low weights we can find up to $m-r-2 \in \{r-2,r-1\}$ nonisomorphic columns whose weights are distinct from the column weights of $X$. Thus we may choose $j\in [1, m-r-2]$ nonisomorphic low weight columns whose weights are not $r-2$ or $r-3$. The result is a low weight matrix $Y$,  and  by Lemma~\ref{lem:extracolumns}, $XY$ is asymmetric. 

\1$\bullet${\bf \boldmath $\left[m-1,2^{m-2} \right]$:} Begin with the $m{\times} (m-1)$ asymmetric binary matrix referred to in the proof of Lemma~\ref{lem:mminusone}.  Recall that it has columns of weight 1 and 2.  For each $j\in [1,2^{m-2}-m+1]$ we wish to find $j$ nonisomorphic low weight columns with weights distinct from the column weights in $X$.  Thus we only need toshow that there exist at least $2^{m-2}-m+1$ non-isomorphic low weight columns of length $m$.

\1There are $\sum_{i=3}^r {m\choose i}$ nonisomorphic columns of weight $w\in [3,r]$.  Note that if $m$ is odd $\sum_{i=0}^r {m \choose i} = 2^{m-1}$, while if $m$ is even $\sum_{i=0}^r {m \choose i} = 2^{m-1}+\frac{1}{2} {m\choose r}$.  In particular, for all $m$, $\sum_{i=0}^r {m \choose i} \geq 2^{m-1}$. Thus $\sum_{i=3}^r {m\choose i} - {m\choose 0} - {m\choose 1} - {m\choose 2} \geq 2^{m-1} - {m\choose 0} - {m\choose 1} - {m\choose 2}$. Thus it is sufficient to show that $2^{m-1} - {m\choose 0} - {m\choose 1} - {m\choose 2}\geq 2^{m-2} -m+1$.    Straightforward algebra shows that this inequality holds if and only if $2^{m-1} \geq m^2-m+4$ which is true for $m\geq 6$. 

\1Thus for each $j\in [1,2^{m-2}-m+1]$ we may choose $j$ nonisomorphic low weight columns of weights distinct from the column weights of $X$, and create an $m\times j$ matrix $Y$.  By Lemma~\ref{lem:extracolumns}, concatenating $X$ with $Y$ produces an asymmetric matrix of the desired dimensions.

\1$\bullet${\bf \boldmath $\left[2^{m-2}, 2^{m-1}-\nu_m \right]$:}  For each $n\in \left[2^{m-2}, 2^{m-1}-\nu_m \right]$, we see that $n'=2^{m-1}-n \in [\nu_m, 2^{m-2}]$.  By the work above, there is an asymmetric $m{\times}n'$ matrix, and then by the Complement Theorem this guarantees the existence of an asymmetric $m{\times}n$ matrix. \end{proof}

\1An entirely similar proof shows that 

\begin{theorem}\label{thm:frommu} \rm For $n\geq 12$ and $m\in [\mu_n, 2^n-\mu_n]$ there exists an asymmetric $m\times n$ matrix.\end{theorem}

\1The existence of the asymmetric $m{\times} n$ matrices for 1) $m\geq 5$, $n\in[\nu_m, 2^{m-1}-\nu_m]$ and 2) $n\geq 4$, $m\in [\mu_n, 2^n-\mu_n]$ provide upper bounds on $\mu_n$ and $\nu_m$ respectively. That is, we have proved the following.

\begin{theorem}\label{thm:upperbd} \rm Let $m\geq 5, n\geq 4$.  If $n\in[\nu_m, 2^{m-1}-\nu_m]$ then $\mu_n\leq m$.\end{theorem}

\begin{theorem}\rm Let $m\geq 5, n\geq 4$.  If $m\in[\mu_n, 2^{n}-\mu_n]$ then $\nu_m\leq n$.\end{theorem}


\section{A Lower Bound}\label{sec:lowerbd}

\1We will now use the tools from Sections \ref{sec:charmat}, \ref{sec:comp}, and \ref{sec:const} to get a lower bound on the cost of 2-distinguishing hypercubes. 

\begin{example}\label{ex:5by4} \rm Since $Q_1, Q_2, Q_3$ are not 2-distinguishable, for all $m\in \N$ there is no $m{\times} n$ asymmetric matrix with $n\in [1,3]$.  In particular, this is true for $m=5$. Thus $\nu_5\geq 4$. Further, by exhaustion there is no $m{\times} 4$ asymmetric matrix for $m\in [1,4]$. Thus $\mu_4 \geq 5$. Since by Bogstad and Cowen \cite{BC} $\rho(Q_4)=5$, there a distinguishing class for $Q_4$ consisting of 5 vertices,  Using  Theorem~\ref{thm:distasymm}, we can conclude that there is an asymmetric $5\times 4$ matrix.  Thus we have $\nu_5=4$ and $\mu_4=5$.\end{example}

\begin{lemma} \label{lem:ineq} \rm For $m\geq 5$, $2^{m-1} - \nu_m < 2^{m} - \nu_{m+1}$. \end{lemma}

\begin{proof}  We will break the proof of this fact into three cases:  $m\in [5,10]$, $m=11$, and $m\geq 12$. 

\1$\bullet${\bf \boldmath $m\in [5,10]$:} By Lemma~\ref{lem:smallm}, for $m\in [5,11]$, $\nu_m=4$.  Thus for any $m\in[5,10]$, $\nu_{m+1}-\nu_{m}= 0$.  This tells us immediately that $2^{m-1}-\nu_{m}< 2^{m}-\nu_{m+1}$ for $m\in[5,10]$. 

\1$\bullet${\bf \boldmath $m =11$:} Again by  Lemma~\ref{lem:smallm}, $\nu_{11}=4$.  Now we will show that $\nu_{12}=5$. By Example~\ref{ex:5by4}, $\mu_5 = 4$.  If there was an asymmetric $12{\times} 4$ matrix, the Complement Theorem would guarantee the existence of an asymmetric $(2^4-12){\times}4$ matrix. But since $\mu_4=\rho(Q_4)=5$, this is not possible.  The same true is for $n\in \{2,3\}$ as argued in Example~\ref{ex:5by4}.  Thus $\nu_{12}\geq 5$. Further we can build an asymmetric $12{\times} 5$ matrix in the following way.  Start with the asymmetric $5{\times} 4$ matrix described in Lemma~\ref{lem:mminusone}.  Find a complementary $5{\times}(2^4-4)$ matrix, $Y$,  whose columns are the 12 low weight columns that are not columns of $X$.  By the Complement Theorem, since $X$ is asymmetric, so is $Y$.  It is easy to check that the rows of $Y$ have weights 3, 4, and 5.  In particular, $Y$ has no rows of weight 6 or more.  Thus by Lemma~\ref{lem:Rules}\ref{lem:Rules:4}, $Y^T$ is an asymmetric $12{\times} 5$ matrix.  Now we can conclude that $\nu_{12}= 5$.  Finally get a proof  of this lemma, we compute $2^{10}-\nu_{11}$ and $2^{11}-\nu_{12}$ to see the strict inequality we are looking for. 

\1$\bullet${\bf \boldmath $m\geq 12$:} Since each of $\nu_m, \nu_{m+1}$ is positive, the magnitude of their difference is less than either of their magnitudes. For $k\geq 12$, by Corollary~\ref{cor:halfs}, there is an asymmetric $k{\times} \lfloor \frac{k}{2} \rfloor$ matrix.  Thus,  $\nu_k\leq \lfloor \frac{k}{2} \rfloor$, and we get that $|\nu_{m+1}-\nu_{m}| < \lfloor \frac{m+1}{2} \rfloor$. Further for $m\geq 3$, $\lfloor \frac{m+1}{2} \rfloor < 2^{m-2}$.  Putting these together we get, $|\nu_{m+1}-\nu_{m}|< 2^{m-1}$.  Simple algebra then yields $2^{m-1}-\nu_{m}<2^{m}-\nu_{m+1}$.  \end{proof}

\begin{corollary}\label{cor:existsunique}\rm For all $n\geq 13$, there exists a unique integer $m$ so that $n\in [2^{m-2}-\nu_{m-1}, 2^{m-1}-\nu_{m}]$. \end{corollary}

\begin{theorem} \label{thm:lowerbd} \rm Let $n\geq 4, m\geq 5$.  If  $n> 2^{m-1}-\nu_{m}$ then $\rho(Q_n) > m$. \end{theorem}

\begin{proof}  By Lemma~\ref{lem:assump}, if $n\geq 2^{m-1}$ then there is no asymmetric $m{\times} n$ binary matrix.  Suppose that $2^{m-1}-\nu_m < n < 2^{m-1}$. Then $n=2^{m-1}-\nu_m + k$ for some $k\in[1,\nu_m-1]$. By the Complement Theorem, the existence of an asymmetric $m{\times} (2^{m-1}-\nu_m + k)$ matrix guarantees the existence of an asymmetric $m{\times}(\nu_m - k)$ matrix. But since $k$ is strictly positive, this contradicts the minimality in the definition of $\nu_m$. Thus there is no asymmetric $m{\times}n$ matrix with $n> 2^{m-1}-\nu_m$.

\1Let $m$ be fixed.  If whenever we have $n> 2^{m-1}-\nu_m$ we are guaranteed that  $n> 2^{k-1}-\nu_k$ for all $k\in [5,m-1]$, then we could conclude that $\mu_n \not\in [5, m]$ and thus that $\mu_n > m$.   By induction, it is enough to show that $2^{m-1}-\nu_{m}\leq 2^{m}-\nu_{m+1}$, which we have by Lemma~\ref{lem:ineq}.

\1Thus when $n> 2^{m-1}-\nu_m$, $\rho(Q_n) > m$.\end{proof}

\1The analogous statements given below have very similar proofs to those above.

\begin{lemma} \label{lem:ineqm} \rm For $n\geq 4$, $2^{n} - \mu_n < 2^{n+1} - \mu_{n+1}$. \end{lemma}

\begin{corollary}\label{cor:existsunique}\rm For all $m\geq 12$, there exists a unique integer $n$ so that $m\in [2^{n-1}-\mu_{n-1}, 2^{n}-\mu_{n}]$. \end{corollary}

\begin{theorem} \label{thm:lowerbd} \rm Let $n\geq 4, m\geq 5$.  If  $m> 2^{n}-\mu_{n}$ then $\nu_m > n$. \end{theorem}


\section{The Distinguishing Cost}\label{sec:distcost}

\1Putting together Theorems~\ref{thm:upperbd} and \ref{thm:lowerbd}, other supporting work, and recalling that $\mu_n=\rho(Q_n)$ we get our final result.

\begin{theorem}\label{thm:formu} \rm  For $n\in [4,12]$, $\rho(Q_n)=5$.  For $n\geq 13$, $\rho(Q_n)=m$ where $m$ is unique integer for which $n\in [2^{m-2}-\nu_{m-1}+1, 2^{m-1}-\nu_m]$.\end{theorem}    

\begin{proof} Lemma~\ref{lem:smalln}, for $n\in[4,12]$, $\mu_n\geq 5$.  By Theorem~\ref{thm:smalln},  $\mu_n\leq 5$.  Thus for $n\in[\nu_5,2^4-\nu_5]$, $\rho(Q_n) = 5$. 

\1By Corollary~\ref{cor:existsunique}, for $n\geq 13$, there exists a unique $m$ so that $n\in [2^{m-2}-\nu_{m-1}+ 1, 2^{m-1} - \mu_m]$.  By Theorem~\ref{thm:upperbd}, given such an $m$, $\rho(Q_n) \leq m$ and by Theorem~\ref{thm:lowerbd}, $\rho(Q_n)>m-1$.  Thus $\rho(Q_n)=m$. \end{proof}

\1A proof entirely similar to that for Theorem~\ref{thm:formu}, gives us the following.  Note that this would require the symmetric statements from Section~\ref{sec:lowerbd}.

\begin{theorem}\label{thm:fornu} \rm  For $m\in [5,11]$, $\nu_m=4$.  For $m\geq 12$, $\nu_m=n$ where $n$ is the unique integer for which $m\in [2^{n-1}-\mu_{n-1}+1, 2^{n}-\mu_n]$.\end{theorem}

\1We can work recursively to compute $\rho(Q_n)$ for any $n\geq 4$ in the following way.  From Theorem~\ref{thm:formu}, we get when $n\in [4, 12], \ \mu_n=5$, and by Theorem~\ref{thm:fornu}, when $m\in [5,11], \ \nu_m=4$.  We can then use Theorem~\ref{thm:formu} repeatedly to get that for $n\in[13, 2^{10}-4]$, $\mu_n= 1+\lceil \log_2 (n+4)\rceil$. For $n>2^{10}-4$, we need $\nu_m$ for $m>11$.  Use   Theorem~\ref{thm:fornu} to find that when $m\in [2^4-\mu_4+1, 2^5-\mu_5], \ \nu_m=5$.  Then we again use Theorem~\ref{thm:formu} to compute $\mu_n= 1+\lceil \log_2 (n+5)\rceil$ for $n\in[2^{10}-3, 2^{26}-5]$.   We can continue in this manner.

\1However, even without specific values for $\nu_m$, we can get a good idea of the value for $\mu_n=\rho(Q_n)$. 

\begin{theorem}\label{thm:oneoftwo} \rm If $n\geq 5$, then $\rho(Q_n) \in \{1 + \lceil \log_2 n\rceil, 2 + \lceil \log_2 n\rceil\}$.  \end{theorem}

\begin{proof}  By computation this is true for $n\in[4,12]$. By Theorem~\ref{thm:formu}, for $n\geq 13$, $\mu_n = \rho(Q_n)=m$ if and only if $2^{m-2}-\nu_{m-1}+1\leq n \leq 2^{m-1}-\nu_m$.  This  is equivalent to $m-1+(\nu_m-\nu_{m-1}) \leq 1 + \log_2(n+\nu_m) \leq m$. Note that $\nu_m-\nu_{m-1}\in \{0,1\}$.   From here it is easy to see that $m=\lceil 1 +  \log_2(n+\nu_m) \rceil = 1 + \lceil \log_2 (n+\nu_m) \rceil$.  Since $m=\mu_n$ here, $1 + \lceil \log_2 (n+\nu_m) \rceil=1 + \lceil \log_2 (n+\nu_{\mu_n}) \rceil$.  By Corollary~\ref{cor:halfs},  $\nu_m \leq \frac{m}{2}$ and $\mu_n \leq \frac{n}{2}$.  Thus $\mu_n \leq 1 + \lceil \log_2 (n+\nu_{\frac{n}{2}}) \leq 1 + \lceil \log_2 (n+\frac{n}{4})\rceil = 1 + \lceil \log_2 \frac{5}{4} + \log_2 n\rceil \approx 1 + \lceil 0.322 + \log_2 n\rceil \in \{1 + \lceil \log_2 n\rceil, 2 + \lceil \log_2 n\rceil\}$.  \end{proof}

\1As discussed in the Introduction, we can sometimes use determining numbers to write our results on distinguishing cost. Since for all $n$, $\Det(Q_n) = 1 + \lceil \log_2 n\rceil$ \cite{Bou2009}, we immediately get the following corollary.

\begin{corollary} For $n\geq 4$, $\rho(Q_n)\in \{\Det(Q_n), \Det(Q_n)+1\}$. \end{corollary}


\section{Open Questions}\label{sec:quest}

\begin{question} \rm Find a more direct way to express the recursive relationship between $n, \mu_n, m$ and $\nu_m$.  \end{question}

\begin{question}\rm Recall that for any 2-distinguishable graph $G$, $\Det(G) \leq \rho(G)$.  Further, we have found that most Cartesian products of prime graphs have  $\rho(G) \in \{\Det(G), \Det(G)+1\}$ \cite{Bou2013a},  and certain Kneser graphs have $\rho(G) = \Det(G)+1$ \cite{Bou2013b}.  Classify the graphs for which $\rho(G) \in \{\Det(G), \Det(G)+1\}$.  Or $\rho(G)= \Det(G)$.  Or $\rho(G) = \Det(G)+1$.\end{question}


\appendix

\section{\boldmath  Existence Proofs for Small Values}\label{ap:small}

\noindent{\bf Lemma~\ref{lem:smalln}:} \rm For $n\in[4,12]$, $\mu_n = 5$.

\begin{proof}   By Lemma~\ref{lem:assump}, there can be no asymmetric $4\times n$ matrix if $n\geq 2^3$.  Since $Q_1, Q_2, Q_3$ are not 2-distinguishable, there is no asymmetric $4\times n$ matrix for $n\in[1,3]$. By the Complement Theorem this implies that there is no asymmetric $4\times n$ matrix for $n\in[5,7]$. Further, one can see by exhaustion that there is no $4\times 4$ asymmetric matrix.  matrix for $n\in [4,6]$.  By the Complement Theorem this implies that there is no asymmetric $4\times n$ matrix for $n\in[4,8]$.  Thus for all $n$, $\mu_n\geq 5$. For $n\in [4,12]$, we will argue the existence of an asymmetric $5{\times}n$ matrix, which will prove that $\mu_n\leq 5$.

\1Let $X$ be the binary matrix in Figure~\ref{fig:5by7}. Suppose that $(\sigma,\varphi)$ is a symmetry of $X$.  Since $X$ is a strictly low weight matrix, by Lemma~\ref{lem:Rules}\ref{lem:Rules:2} and~\ref{lem:Rules}\ref{lem:Rules:3}, $\varphi$ and $\sigma$ can only permute columns and rows of the same weight.  Since the $1\st$ column is the only one of weight 1, it is fixed by $\varphi$.  This forces the $1\st$ row to have a 1 in its $1\st$ position.   Since the $1\st$ row is the only row with a 1 in its $1\st$ position, the $1\st$ row is fixed by $\sigma$.   Since the $3\rd$ row is the only one of weight 4 and the $5\th$ row is the only one of weight 2, these rows are fixed by $\sigma$. Since the $2\nd$ column is the only one with 1,0,0 in its $1\st, 3\rd, 5\th$ positions respectively (coming from fixed rows 1, 3, and 5), the $2\nd$ column is fixed by $\varphi$. Since the $2\nd$ row is the only not-yet-known-to-be-fixed row with a 1 in position 2, (coming from the fixed $2\nd$ column), the $2\nd$ row is fixed by $\sigma$. Now that we know that all of rows $1,2,3,5$ are fixed, the $4\th$ row must be fixed by $\sigma$ as well.  Thus $\sigma$ is trivial.  Since $X$ has no identical rows and no isomorphic columns, by Lemma~\ref{lem:iff}, $\varphi$ is trivial as well.  Thus $X$ is asymmetric.

\begin{figure}[htb] \begin{center}$\left[\begin{array}{cccccccc} 
1 & 1 & 0 & 0 & 0 & 1 & 0 & 0\\
0 & 1 & 1 & 0 & 0 & 0 & 0 & 1 \\
0 & 0 & 1 & 1 & 0 & 1 & 1 & 0 \\
0 & 0 & 0 & 1 & 1 & 0 & 0 & 1 \\
0 & 0 & 0 & 0 & 1 & 0 & 1 & 0
\end{array} \right]$\end{center} 
\caption{$5{\times} 7$ asymmetric matrix, $X$}
\label{fig:5by7}
\end{figure}

\1Arguments similar to the one above show that for $n\in[4,7]$ the submatrix consisting of the first $n$ columns of $X$ is asymmetric. Using these $5{\times}n$ asymmetric matrices, the Complement Theorem guarantees $5{\times} n$ asymmetric matrices for $n\in [2^4-7, 2^4-4]$.  Together, these prove the existence of asymmetric $5{\times}n$ matrices with $n\in [4,2^4-4]$ as desired.\end{proof}

\noindent{\bf Lemma~\ref{lem:smallm}:} \rm For $m\in[5,11]$, $\nu_m = 4$.

\begin{proof} As argued in Example~\ref{ex:5by4}, for all $m$, $\nu_m\geq 4$. The matrices  in Figure~\ref{fig:smallm} prove the existence of asymmetric $m{\times}4$ matrices for $m\in [5,8]$. (The reader may check that these matrices are indeed asymmetric. Start by looking at which columns must be fixed by their unique column weight.  For those with no column of weight $\frac{m}{2}$, Lemma~\ref{lem:Rules}\ref{lem:Rules:3} can be used to do the same with row weights.  The rest takes only a little logic.) Using these matrices, the Complement Theorem guarantees the existence of $m\times 4$ asymmetric matrices for $m\in[8,11]$. 
Thus we have proved that for $m\in[5,11]$, $\nu_m\leq 4$ and we have the desired equality. 
\end{proof}

\begin{figure}[htb]

\begin{center}$\left[\begin{array}{cccc} 
1 & 1 & 0 & 0 \\
0 & 1 & 1 & 0 \\
0 & 0 & 1 & 1 \\
0 & 0 & 0 & 1 \\
0 & 0 & 0 & 0
\end{array} \right]$\hspace{1in}
$\left[\begin{array}{cccc} 
1 & 1 & 0 & 0 \\
0 & 1 & 1 & 0 \\
0 & 0 & 1 & 1 \\
0 & 0 & 0 & 1 \\
0 & 0 & 1 & 0 \\
0 & 0 & 0 & 0 
\end{array} \right]$\end{center}

\vskip.25in

\begin{center}
$\left[\begin{array}{cccc} 
1 & 1 & 0 & 0 \\
0 & 1 & 1 & 0 \\
0 & 0 & 1 & 1 \\
0 & 0 & 0 & 1 \\
0 & 0 & 1 & 0 \\
0 & 1 & 0 & 0 \\
0 & 0 & 0 & 0 

\end{array} \right]$\hspace{1in}
$\left[\begin{array}{cccc} 
1 & 1 & 0 & 0 \\
0 & 1 & 0 & 0 \\
1 & 0 & 0 & 0 \\
0 & 0 & 0 & 1 \\
1 & 0 & 1 & 0 \\
1 & 1 & 1 & 0 \\
0 & 1 & 1 & 1 \\
0 & 0 & 0 & 0 
\end{array} \right]$\end{center}
\caption{$m\times 4$ asymmetric matrices for $m\in [5,8]$.}
\label{fig:smallm}
\end{figure}

\1{\bf Theorem~\ref{thm:smallm}:}  For $m\in [5, 11]$, and $n\in [4,2^{m-1}-4]$ there exists an asymmetric $m{\times}n$ matrix. 

\begin{proof} For each $m\in[5,11]$ we construct an asymmetric $m{\times} n$ matrix for $n\in [4,2^{m-1}-4]$.   The proof is broken into three cases based on the the size of $m$.  These cases are $m=5$, $m=6$, and $m\in[7,11]$.

\1$\bullet${\boldmath $m=5$:}  The truth of this case is proved in Lemma~\ref{lem:smalln}.

\1$\bullet${\boldmath $m=6$:}   By Lemma~\ref{lem:smallm}, we know that there exists an asymmetric $6{\times} 4$ matrix.  Further by Lemma~\ref{lem:mminusone},  there exist asymmetric $6{\times} 5$ and $6{\times}6$ matrices with columns of weight 1 and 2. Note that there are $\frac{1}{2}{6\choose 3} = 10$ binary strings of length 6 and weight 3. For $j\in[1,10]$ create a $6{\times}j$ matrix with nonisomorphic columns of weight 3. Concatenate the result with the asymmetric $6\times 6$ matrix with columns of weight 1 and 2 identified above. By Lemma~\ref{lem:extracolumns}, the resulting $6{\times}(4+j)$ matrix is asymmetric.   Thus there are asymmetric $6{\times} n$ matrices for $n\in [4,2^4]$.  Using the Complement Theorem, this guarantees the existence of asymmetric $6{\times} n$ matrices for $n\in[2^5-16,2^5-4]$.  Overall, we now have proof of existence of asymmetric $6\times n$ matrices for $n\in[4,2^5-4]$ as desired.

\1$\bullet${\boldmath $m\in[7,11]$:}  For $m\in [7,11]$, we know by Lemma~\ref{lem:smallm} that there exists an asymmetric $m{\times} 4$ matrix, $X$ . Next we will show that for $n\in [5,m-1]$ there are asymmetric $m{\times}n$ matrices.

\1We can achieve an asymmetric $7{\times} 5$ matrix by concatenating a $7\times1$ matrix of zeros to the $7{\times} 4$ matrix in Figure~\ref{fig:smallm} without disturbing the asymmetry.  Further by Lemma~\ref{lem:mminusone}, we can achieve an asymmetric $7{\times}6$ matrix.

\1The asymmetric $8\times 4$ matrix $X$ given in  Lemma~\ref{lem:smallm} has no column of weight 1.  Thus for $j\in [1,3]$, we can concatenate an $m{\times} j$ matrix with distinct columns of weight 1 to $X$. By Lemma~\ref{lem:smallm}, the resulting $8{\times} n$ matrices with $n\in[5,7]$ are asymmetric. 

\1For each of $m\in[9,11]$, we achieve an $m{\times} 4$ matrix $X$ by using the Complement Theorem on  the $(16-m){\times} 4$ matrix from the proof of Lemma \ref{lem:smallm} and then using the nontrivial automorphism of $\Z_2$ on each column.  It is easy to check that each such $X$, has  no column of weight 1.  Thus for $j\in [1,m-5]$, we can concatenate an $m{\times} j$ matrix consisting of distinct columns of weight 1 with the $m\times 4$ matrix.  By Lemma~\ref{lem:extracolumns} the resulting $9{\times} n$ matrices with $n\in[5,8]$ are asymmetric. 

\1Thus we now have asymmetric $m{\times} n$ matrices for $m\in[7,11]$ and $n\in[4,m-1]$.  Now to complete this to $n\in[4,2^{m-2}]$. 
 
\1For $m\geq 5$, by Lemma~\ref{lem:mminusone},  there exist asymmetric $m{\times}m$ matrix with columns of weight 1 and 2. Choose such a matrix for each $m$ and denote it by $X$.  One can easily check that for $m\in[7,11]$, ${m\choose 3}\leq 2^{m-2}-m$.    Since for each of these $m$,  $3<\frac{m}{2}$, weight 3 columns of length $m$ are necessarily nonisomorphic.  Thus for each $j\in[1,2^{m-2}-m]$ we can construct an $m{\times} j$ matrix $Y$ of distinct weight 3 columns. Concatenate $Y$  with $X$ to construct an $m{\times}(m+j)$ matrix.   By Lemma~\ref{lem:extracolumns}, the concatenated matrix is asymmetric.  

\1Thus we have proved the existence of asymmetric $m{\times} n$ matrices for $m\in [5,11]$ and $n\in [4, 2^{m-2}]$.  Using the Complement Theorem these asymmetric matrices guarantee the existence of asymmetric $m{\times} n$  matrices for $n\in [2^{m-2}, 2^{m-1}-4]$.\end{proof}

\1An entirely similar proof gives us the analogous result below.

\1{\bf Theorem~\ref{thm:smalln}:}  For $n\in [4, 12]$, and $m\in [5,2^n-5]$ there exists an asymmetric $m{\times}n$ matrix.

\bibliographystyle{plain}
\bibliography{HypercubeCost}

\end{document}